\documentclass[10pt]{article}

\usepackage{bigints}

\usepackage{amssymb}
\usepackage{mathrsfs}
\usepackage{amsmath}
\usepackage{amsfonts}

\renewcommand{\thefootnote}{\fnsymbol{footnote}}

\usepackage{amsthm}
\newtheorem{theo}{Theorem}[section]
\newtheorem{prop}{Proposition}[section]
\newtheorem{coro}{Corollary}[section]
\newtheorem{lemm}{Lemma}[section]
\newtheorem{fact}{Fact}[section]

\theoremstyle{definition}

\theoremstyle{remark}
\newtheorem{exam}{Example}[section]
\newtheorem{rema}{Remark}[section]
\newtheorem{ques}{Question}[section]

\usepackage{hyperref}

\title{Notes on scalar curvature lower bounds of steady gradient Ricci solitons}
\author{Shota Hamanaka}
\date{\today}

\begin{document}
\maketitle
 
\renewcommand{\thefootnote}{\fnsymbol{footnote}} 
\renewcommand{\thefootnote}{\arabic{footnote}}

\renewcommand{\thefootnote}{\fnsymbol{footnote}} 
\renewcommand{\thefootnote}{\arabic{footnote}}
  
\begin{abstract}
We provide new type of decay estimate for scalar curvatures of steady gradient Ricci solitons.
We also give certain upper bound for the diameter of a Riemannian manifold whose $\infty$-Bakry--{\'E}mery Ricci tensor is bounded by some positive constant from below.
For the proofs, we use $\mu$-bubbles introduced by Gromov.
\end{abstract}

\section{Introduction}\label{sec1}
A triple $(M, g, f)$ consisting of a Riemannian manifold $(M, g)$ and a function $f$ on $M$ is called \textit{steady gradient Ricci soliton} if 
\[
\mathrm{Ric}_{g} + \mathrm{Hess}_{g} f = 0,
\]
where $\mathrm{Ric}_{g}$ and $\mathrm{Hess}_{g} f$ denote respectively the Ricci tensor of $g$ and the Hessian of $f$ with respect to $g.$
Hamilton have proven \cite{hamilton1993formations} that
\begin{equation}
\label{eq-1}
  R_{g} + |\nabla f|_{g}^{2} = C_{0}
\end{equation}
for some constant $C_{0}.$ Here, $R_{g}$ denotes the scalar curvature of $g.$
Moreover, by {\cite[Theorem 1.3]{zhang2009completeness}}, $R_{g} \ge 0.$
Steady Ricci solitons appear as Type II singularity models of Ricci flows with nonnegative curvature operator and positive Ricci curvature (see \cite{hamilton1993eternal, hamilton1993formations} and {\cite[Theorem 3.4]{cao2009recent}}).
First of all, from {\cite[Theorem 1.7]{munteanu2013gradient}} and {\cite[Theorem 2.2]{karp1985differential}}, the following holds.
\begin{fact}
  \label{theo-1}
  Let $(M, g)$ be a complete steady gradient Ricci soliton. 
  Then $\inf_{M} R_{g} = 0.$
\end{fact}
  \begin{proof}
    From (\ref{eq-1}), 
    \[
    \limsup_{d_{g} (p, x) \rightarrow \infty} \frac{- f(x)}{d_{g}(p, x)} < +\infty
    \]
    for some fixed point $p \in M.$
    Moreover, from {\cite[Theorem 1.7]{munteanu2013gradient}}, 
    \[
    \limsup_{r \rightarrow +\infty} \frac{\log \mathrm{Vol}_{g} ( B_{g}(p, r))}{r} = 0,
    \]
    where $\mathrm{Vol}( B_{g}(p, r) )$ denotes the volume of the geodesic ball $B_{g}(p, r)$ centered at $p \in M$ of radius $r > 0$ with respect to $g.$
    Hence from the traced soliton identity: $R_{g} + \Delta_{g} f$ = 0 and {\cite[Theorem 2.2]{karp1985differential}}, 
    \[
    \inf_{M} R_{g} = \inf_{M} \Delta_{g} (-f) \le 0.
    \]
    See \cite{chow2011mild, fernandez2011maximum, wu2013potential} for other proofs.
  \end{proof}
  \begin{rema}
  This has been proven by Fern{\'a}ndez-L{\'o}pez--Garc{\'\i}a-R{\'\i}o \cite{fernandez2011maximum} using Omori--Yau maximum principle.
  Later, Wu \cite{wu2013potential} has proven this fact as a corollary of certain decay estimate of the potential function.
    Moreover Chow--Lu \cite{chow2011mild} have shown that $\inf_{M} |\mathrm{Ric}_{g}| = 0.$
  \end{rema}
Fact \ref{theo-1} has some similarities in mean curvature flows.
In \cite{ma2019volume}, Ma proved that if $M$ is a complete convex translating soliton in $\mathbb{R}^{n+1},$ then its infimum of the scalar mean curvature is zero (see {\cite[Theorem 1~(3)]{ma2019volume}}).
For convexity of translating solitons, see also the recent work by Xie--Yu \cite{xie2023convexity}.
And, Sun \cite{sun2014mean} proved that for a (real) two-dimensional complete symplectic (resp. almost--calibrated Lagrangian) translating soliton $\Sigma$ in $\mathbb{C}^{2}$ with quadratic area growth, if its K\"{a}hler angle is not too large, then its norm of the mean curvature vector must be zero.
  
In this note, we provide another proof (using $\mu$-bubbles) of Fact \ref{theo-1} under certain much stronger condition.
Beyond Fact \ref{theo-1}, we also investigate the decay of the scalar curvature by using $\mu$-bubbles.
Wu {\cite[Corollary 1.3]{wu2013potential}} has investigated the order of decay of $\inf_{M} R_{g}.$
We give another type of decay estimate of $\inf_{M} R_{g}$ in the following our main theorem.

The main result of this note is the following theorem.
\begin{theo}
\label{main-3}
  Let $n \ge 2$ and $0 < \alpha \le 1$. Suppose that $(M^{n}, g, f)$ is an $n$-dimensional complete non-compact nonparabolic steady gradient Ricci soliton.
  Assume that 
  \begin{equation}
    \label{eq-2}
    \lim_{d_{g}(p, x) \rightarrow \infty} G(x) \cdot d^{\alpha}_{g}(p, x) = 0
  \end{equation}
  where $G(\cdot)$ is the minimal positive Green's function with the pole at $p \in M.$
  Moreover we assume that there is a constant $C \ge 0$ such that
  \begin{equation}
    \label{eq-ric2}
    \mathrm{Ric}_{g} (x) \ge -C d_{g}(p, x)^{-2} \cdot G(x)
  \end{equation}
  for all $x \in M$ with $d_{g} (p, x) >> 1.$
  Then there is a positive constant $C = C(n) > 0$ such that 
  \[
  \underset{x \rightarrow \infty}{\liminf}\, R_{g}(x)\, d^{\alpha}_{g}(x, p)
  \begin{cases}
    \le C C^{1/2}_{0} &\mathrm{if}~\alpha = 1, \\
    = 0 &\mathrm{if}~0 < \alpha < 1.
  \end{cases}
  \]
  Here $C_{0}$ denotes the constant in the equation (\ref{eq-1}).
\end{theo} 
\begin{rema}
  The assumption (\ref{eq-ric2}) can be replaced with 
  \[
  \mathrm{Ric}_{g}(x) \ge -C d^{-2}_{g}(p, x)~~\mathrm{and}~\lim_{d_{g}(p, x) \rightarrow +\infty} d^{2-\alpha}_{g}(p, x) \cdot G(x) = +\infty.
  \]
  See Remark \ref{rema-main-3} below.
\end{rema}
 Recall that a complete Riemannian manifold $(M, g)$ is \textit{nonparabolic} if it admits a positive symmetric Green's function.
 From Propositions \ref{prop-harmonic} and \ref{prop-green} below, we can immediately obtain the following corollary.
 \begin{coro}
   Let $n \ge 2.$ Suppose that $(M^{n}, g, f)$ is an $n$-dimensional complete non-compact nonparabolic steady gradient Ricci soliton with $\mathrm{Ric}_{g} \ge 0.$
  Assume that 
  \[
   \lim_{d_{g}(p, x) \rightarrow \infty} G(x) \cdot d_{g}(p, x) = 0
  \]
  where $G(\cdot)$ is the minimal positive Green's function with the pole at $p \in M.$
  Then there is a positive constant $C = C(n) > 0$ such that 
  \[
  \underset{x \rightarrow \infty}{\liminf}\, R_{g}(x) d_{g}(x, p) \le C C^{1/2}_{0}.
  \]
  Here $C_{0}$ denotes the constant in the equation (\ref{eq-1}).
 \end{coro}
 \begin{rema}
Suppose that $(M^{n}, g, f)$ is an $n$-dimensional complete non-compact steady gradient Ricci soliton with $\mathrm{Ric}_{g} \ge 0.$
If there is a point $p \in M$ such that $\mathrm{Ric}_{g}(p) = 0,$ then $(M, g)$ is Ricci flat (see {\cite[Remark 6.58]{chow2023hamilton}}).
\end{rema}
\begin{rema}
  Munteanu--Wang have recently proven {\cite[Lemma 2.4]{munteanu2024geometry}} that $C$ can be taken to be zero in dimension three without assuming that $(M, g, f)$ is a steady gradient Ricci soliton and the assumption (\ref{eq-2}).
\end{rema}
This paper is organized as follows.
We prove our main theorem (Theorem \ref{main-3}) in Subsection \ref{subsection-3}.
Before proving Theorem \ref{main-3}, we provide some propositions in Section \ref{section-propositions}.
The proofs of Propositions \ref{main-1}, \ref{main-2} are given respectively in subsections \ref{subsection-1}, \ref{subsection-2} as a prelude to proving our main theorem.
In Section \ref{section-pre}, we provide some preliminaries that is key tools to prove our Propositions \ref{main-1}, \ref{main-2} and Theorem \ref{main-3}.
In Appendix ($=$ Section \ref{section-appendix}), we give a (partially) new upper bound of the diameter of a Riemannian manifold whose $\infty$-Bakry--{\'E}mery Ricci tensor is bounded by some positive constant from below.

\subsection*{Acknowledgements}
\,\,\,\,\,\,\, Part of this paper was written during a short stay at Mathematics M\"{u}nster.
The author would be grateful to Prof. Rudolf Zeidler and Ms. Claudia--Maria R\"{u}diger their hospitality and kind support.
The author would like to thank Prof. Keita Kunikawa for notifying him the paper \cite{sun2014mean}.
The author also would like to thank Prof. Homare Tadano for notifying him the paper \cite{wu2018myers}.
Finally, the author also would like to thank Prof. Yohei Sakurai for asking a question about our theorems in the cases of $n \ge 8$.
This work was supported by JSPS KAKENHI Grant Number 24KJ0153.

\section{Propositions}
\label{section-propositions}
In this section and Section \ref{section-proof}, we provide some propositions which are weaker than Fact \ref{theo-1} and prove them as a prelude to proving our main theorem (Theorem \ref{main-3}).

For a steady gradient Ricci soliton $(M, g, f),$ from (\ref{eq-1}) and $R_{g} \ge 0$ ({\cite[Theorem 1.3]{zhang2009completeness}}), it turns out that $C_{0}^{-1}\, \underset{M}{\inf}\, R_{g}$ ($C_{0}$ is the constant in (\ref{eq-1})) is a scale invariant non-negative constant.
Here, ``scale invariant" means that this quantity is invariant under the scaling: $g \mapsto c \cdot g,$ where $c > 0$ is some positive constant.
\begin{prop}
  \label{main-1}
Let $n \ge 2.$ Suppose that $(M^{n}, g, f)$ is an $n$-dimensional complete non-compact steady gradient Ricci soliton.
Assume that there is a constant $C \ge 0$ and a point $p \in M$ such that
\begin{equation}
  \label{eq-ric1}
\mathrm{Ric}_{g}(x) \ge - C d_{g}(p, x)^{-\alpha} 
\end{equation}
for all $x \in M$ with $d_{g}(p, x) >> 1$ for some positive constant $\alpha > 0.$
Then
\[
  C_{0}^{-1}\, \underset{M}{\inf}\, R_{g} \le A(n),
\]
  where
  \[
  A(n) := \frac{\frac{4}{3} + \frac{(n-3)(n+1)}{4(n-1)}}{\frac{7}{3} + \frac{(n-3)(n+1)}{4(n-1)}}.
  \]
\end{prop}
  \begin{prop}
\label{main-2}
  Let $n \ge 2.$ Suppose that $(M^{n}, g, f)$ is an $n$-dimensional complete non-compact nonparabolic steady gradient Ricci soliton.
  Assume that $G(x) \rightarrow 0$ as $d_{g}(p, x) \rightarrow \infty$ (i.e., for any $\varepsilon > 0,$ there is $\delta = \delta(\varepsilon) > 0$ such that $G(x) < \varepsilon$ for all $x \in M$ with $d_{g}(p, x) \ge \delta$), and
  there is a constant $C \ge 0$ such that
  \begin{equation}
    \label{eq-ric3}
    \mathrm{Ric}_{g}(x) \ge -C d_{g}(p, x)^{-2} \cdot G(x)
  \end{equation}
  for all $x \in M$ with $d_{g} (p, x) >> 1.$
  Here, $G(\cdot)$ is the minimal positive Green's function with the pole at $p \in M.$
  Then $\underset{M}{\inf}\, R_{g} = 0.$
  \end{prop}
  \begin{rema}
  \begin{itemize}
  \item[(1)] The assumption (\ref{eq-ric3}) can be replaced with 
  \[
  \mathrm{Ric}_{g}(x) \ge -C d^{-2}_{g}(p, x)~~\mathrm{and}~\lim_{d_{g}(p, x) \rightarrow +\infty} d_{g}(p, x) \cdot G(x) = +\infty.
  \]
  See Remark \ref{rema-2.2} below.
    \item[(2)] From Propositions \ref{prop-harmonic} and \ref{prop-green} below, the assumption of Proposition \ref{main-2} is satisfied if $\mathrm{Ric}_{g} \ge 0.$ 
    \end{itemize}
  \end{rema}
  \begin{rema}
    Munteanu--Wang have recently proven the same statement for dimension 3 without assuming that $(M, g, f)$ is a steady gradient Ricci soliton {\cite[Theorem 3.5]{munteanu2024geometry}}.
    Their proof is based on their analysis of certain harmonic functions and it relies on the fact that the manifold is three dimensional.
    Our proof is instead relies on $\mu$-bubbles introduced by Gromov.
  \end{rema}
    A complete non-compact Riemannian manifold is nonparabolic if and only if it has at least one nonparabolic end. (See immediately after Definition 20.5 in \cite{li2012geometric}.)
    Munteanu--Sesum proved in {\cite[Theorem 1.5]{munteanu2013gradient}} that any steady gradient Ricci solotion has at most one nonparabolic end.
    
  From the recent result of Bamler--Chan--Ma--Zhang {\cite[Theorem 1.1]{bamler2023optimal}} and the criterion (\ref{eq-criterion}) below, if $(M^{n}, g, f)$ is a complete steady gradient Ricci soliton with $n \ge 4,$ $\mathrm{Ric}_{g} \ge 0$ and the corresponding Ricci flow $(M, g_{t})_{t \in \mathbb{R}}$ has a uniformly bounded Nash entropy (see the condition $(1.4)$ in \cite{bamler2023optimal}), which is either 
\begin{itemize}
  \item $(M, g_{t})_{t \in \mathbb{R}}$ arises as a singularity model, or
  \item $(M, g)$ has bounded curvature,
\end{itemize}
then $(M, g)$ is nonparabolic.

\begin{exam}
  The Cigar soliton is the unique two dimensional complete steady gradient Ricci soliton with positive Gaussian curvature. 
  Moreover, since the Cigar soliton has linear volume growth, it is parabolic (i.e., it is not nonparabolic).
\end{exam}

\begin{exam}
  $(\mathrm{Cigar} \times \mathbb{R}, g_{\mathrm{Cigar}} + g_{eucl})$ is a three dimensional steady gradient Ricci soliton with nonnegative curvature.
  Since the Cigar soliton has linear volume growth, it turns out that $\mathrm{Cigar} \times \mathbb{R}$ has quadratic volume growth (by the coarea formula).
  Hence $\mathrm{Cigar} \times \mathbb{R}$ is parabolic.
\end{exam}
\begin{exam}
  The $n$-dimensional Bryant solitons on $\mathbb{R}^{n}~(n \ge 3)$ are rotationally symmetric and positive sectional curvature.
  The volume of geodesic balls $B_{r}(o)$ grow on the order $r^{\frac{n+1}{2}},$
  and the curvature decay is of $O(r^{-1}).$
  In particular, the $n$-dimensional Bryant soliton satisfies $\lim_{d(p, x) \rightarrow +\infty} R(x) = 0,$ and
  it is parabolic if $n = 3$ and nonparabolic if $n \ge 4.$
  \end{exam}
  \begin{exam}
Lai \cite{lai2020family} recently found a family of $n$-dimensional $(n \ge 3)$ steady gradient Ricci soliton which is $\mathbb{Z}_{2} \times O(n-1)$-symmetric but not rotationally symmetric with positive curvature operator.
Moreover, she \cite{lai20223d} has also proven that there exists a $\mathbb{Z}_{2} \times O(2)$-symmetric three dimensional flying wing which is asymptotic to a sector with angle $\theta$ for all $\theta \in (0, \pi).$
(Note that $\theta = 0, \pi$ respectively corresponds to the Bryant soliton and $\mathrm{Cigar} \times \mathbb{R}.$)
From {\cite[Corollary 1.5]{lai2020family}}, for such a three-dimensional flying wing $(M, g, f),$ $\inf_{M} R_{g} = 0.$
    \begin{ques}
      Are such three-dimensional flying wings nonparabolic?
    \end{ques}
  \end{exam}

\section{Preliminaries}
\label{section-pre}
\subsection{Green's functions on a Riemannian manifold with Ricci lower bound}
Let $(M^{n}, g)$ be a smooth complete non-compact Riemannian manifold.
Recall that it is called \textit{nonparabolic} if it admits a positive symmetric Green's function.
It is well known that in this case the minimal positive Green's function $G(x, y)$ may be obtained as the limit of the Dirichlet Green's function of a sequence of compact exhaustive domains of the manifold.
Then,
\[
\Delta_{x} G(x, y) = - \delta_{x} (y),~~G(x, y) = G(y, x) > 0.
\]
The first key tool used to prove our main theorems is the following gradient estimate for positive harmonic functions by Li--Wang \cite{li2002complete} (see also \cite{li2012geometric}).
\begin{prop}[{\cite[Lemma 2.1]{li2002complete}} or {\cite[Theorem 6.1]{li2012geometric}}]
\label{prop-harmonic}
  Let $(M^{n}, g)$ be a complete Riemannian manifold.
  Suppose that $h$ is a positive harmonic function defined on the geodesic ball $B_{g} (p, 2R) \subset M$ of radius $2R$ centered at $p$ and $B_{g} (p, 2R) \cap \partial M = \emptyset,$ and
  \[
  \mathrm{Ric}_{g} \ge -(n-1) \rho^{2}
  \]
  for some constant $\rho \in \mathbb{R}.$
  Then there is a positive constant $C = C(n) > 0$ such that
  \[
  \frac{|\nabla h|^{2}_{g}(x)}{h(x)^{2}} \le C (1 + \varepsilon^{-1}) R^{-2} + \frac{(4(n-1)^{2} + 2\varepsilon) \rho^{2}}{4-2\varepsilon}
  \]
  for all $x \in B_{g}(p, R)$ and for any $\varepsilon < 2.$
\end{prop}
 By {\cite[Corollary of Theorem 2]{varopoulos1981poisson}}, a complete non-compact Riemannian manifold $(M^{n}, g)$ with $\mathrm{Ric}_{g} \ge 0$ is nonparabolic if and only if 
  \begin{equation}
  \label{eq-criterion}
     \int_{1}^{\infty} \frac{t}{\mathrm{Vol}_{g} (B_{g}(p, t))}\, dt < +\infty,
  \end{equation}
  where $\mathrm{Vol}_{g} (B_{g}(p, t))$ is the volume of geodesic ball $B_{g}(p, t)$ of radius $t$ centered at $p$ with respect to $g.$
  The second key tool is the following estimate of the minimal positive Green's function $G(x, y)$ by Li--Yau {\cite[Theorem 5.2]{li1986parabolic}}.
  \begin{prop}[{\cite[Theorem 5.2]{li1986parabolic}}]
  \label{prop-green}
     Suppose that $(M^{n}, g)$ is a complete nonparabolic Riemannian manifold with $\mathrm{Ric}_{g} \ge 0.$
     Then the minimal positive Green's function $G(x, y)$ satisfies that there is a positive constant $C = C(n) > 0$ such that
     \[
     C^{-1} \int_{d_{g}(x, y)}^{\infty} \frac{t}{\mathrm{Vol}_{g} (B_{g}(x, t))}\, dt \le G(x, y) \le C \int_{d_{g}(x, y)}^{\infty} \frac{t}{\mathrm{Vol}_{g} (B_{g}(x, t))}\, dt
    \]
    for all $x \neq y.$
    In particular, $G(x, y) \rightarrow 0$ as $d_{g}(x, y) \rightarrow \infty.$
  \end{prop}

\subsection{Warped $\mu$-bubbles}
Let $(X, g)$ be an oriented connected Riemannian manifold together with a decomposition 
  $\partial X = \partial_{-} X \sqcup \partial_{+} X,$
  where $\partial_{\pm} X$ are (non-empty) unions of boundary components.
  Fix a smooth function $u > 0$ on $X$ and a smooth function $h$ on the interior $\mathring{X}.$
Choose a Caccioppoli set $\Omega_{0}$ with smooth boundary, which contains an open neighborhood of $\partial_{-} X$ and is disjoint from $\partial_{+} X.$
Consider the following functional
\[
\mathcal{A}_{u, h} (\Omega) = \int_{\partial^{*} \Omega} u\, d \mathcal{H}^{n-1} - \int_{X} (\chi_{\Omega} - \chi_{\Omega_{0}}) h u\, d\mathcal{H}^{n}
\]
for all Caccioppoli sets $\Omega$ with $\Omega \Delta \Omega_{0} \Subset \mathring{X}.$
Here, $\mathcal{H}^{k}$ denotes the $k$-Hausdorff measure with respect to the distance $d_{g}$ induced from $g.$
$\mathcal{C}(X)$ denotes the set of all Caccioppoli sets $\Omega$ such that $\Omega \Delta \Omega_{0} \Subset \mathring{X}.$
If $\Omega \in \mathcal{C}(X)$, $\Omega$ contains an open neighborhood of $\partial_{-} X$ and is disjoint from $\partial_{+} X.$
A Caccioppoli set minimizing $\mathcal{A}_{u, h}$ in this class $\mathcal{C}(X)$ is called a \textit{warped} $\mathit{\mu}$-\textit{bubble}.
The existence and regularity of a minimizer of $\mathcal{A}_{u, h}$ was given in {\cite[Proposition 12]{chodosh2024generalized}} (see also {\cite[Proposition 2.1]{zhu2021width}}) and {\cite[Theorem 2.2]{zhou2020existence}}.
\begin{prop}[\cite{bellettini2019stable, aiex2026quantitative} (see also {\cite[Proposition 12]{chodosh2024generalized}} and {\cite[Theorem 2.2]{zhou2020existence}})]
\label{prop-existence}
  Let $n \ge 2.$ Suppose that $h(x) \rightarrow \pm \infty$ as $x \rightarrow \partial_{\mp} X.$
  Then there is a warped $\mu$-bubble $\Omega$ such that $\Omega \Delta \Omega_{0} \Subset \mathring{X}.$
  Moreover, the regular part $\partial^{\mathrm{reg}} \Omega$ of $\partial \Omega$ is a smooth hypersurface and the singular part $\partial^{\mathrm{sing}} \Omega$ of $\partial \Omega$ has Hausdorff dimension at most $n-8.$
\end{prop}
\begin{rema}
  The statement about regularity stated in \cite{bellettini2019stable, aiex2026quantitative} was only for the case $u \equiv 1$.
  However, the corresponding statement for general smooth function $u > 0$ can also be obtained as follows.
  As mentioned in {\cite[Remark 11]{chodosh2024generalized}}, the variational problem for $\mathcal{A}_{u, h}$ inside $X$ is equivalent to the corresponding functional $\tilde{\mathcal{A}}_{\tilde{h}}$ among the $S^{1}$-invarinat Caccioppoli sets inside $(X \times S^{1}, g + u^{2} dt^{2})$.
  Suppose that the Hausdorff dimension of the singular set $\partial^{\mathrm{sing}} \Omega$ of a minimizer $\Omega$ is greater than $n-8$.
  Then $\partial^{\mathrm{sing}} \Omega \times S^{1}$ is contained in the singular set of the minimizer $\Omega \times S^{1}$ of $\tilde{\mathcal{A}}_{\tilde{h}}$, and its Hausdorff dimension is greater than $(n-8) + 1 = n-7$.
  On the other hand, from {\cite[Theorem 1.1]{bellettini2019stable}}, the Hausdorff dimension of $\partial^{\mathrm{sing}} (\Omega \times S^{1})$ is at most $(n+1)-8 = n-7$.
  This is a contradiction.
\end{rema}
The first and second variation formula are given in {\cite[Lemma 13]{chodosh2024generalized}} and {\cite[Lemma 14]{chodosh2024generalized}} respectively (see also {\cite[4.1, 4.3]{rade2023scalar}}). We use the the second variation formula in the form of {\cite[Theorem 4.3]{chodosh2401stable}}.
\begin{prop}[First variation {\cite[Lemma 4.10]{rade2023scalar}}]
\label{prop-first}
Suppose $\Omega \in \mathcal{C}(X)$ is a smooth and let $\Sigma$ be a connected component of $\partial \Omega \setminus \partial_{-} X.$
  For any smooth function $\phi$ on $\Sigma$ let $V_{\phi}$ be a vector field on $X,$ which vanishes outside a small neighborhood of $\Sigma$ and agree with $\phi \nu$ on $\Sigma.$ Here, $\nu$ is the outwards pointing unit normal of $\Sigma.$ Let $\Phi_{t}$ be the flow generated by $V_{\phi}$ with $\Phi_{0} = \mathrm{id}.$
  Then
  \[
  \left. \frac{d}{dt} \right|_{t = 0} \mathcal{A}_{u, h} (\Phi_{t} (\Omega)) = \int_{\Sigma} (H u + g(\nabla^{X} u, \nu) - hu) \phi\, d\mathcal{H}^{n-1}.
  \]
  Here, $\nabla^{X} u \in \mathfrak{X}(X)$ denotes the gradient vector field of $u$ with respect to $g$ and $H$ is the scalar mean curvature of $\Sigma.$
  In particular, a smooth $\mu$-bubble $\Omega$ satisfies 
  \begin{equation}
  \label{eq-3}
    H = -u^{-1} g (\nabla^{X} u, \nu) + h
  \end{equation}
  along $\partial \Omega.$
\end{prop}

\begin{prop}[Second variation {\cite[In the proof of Theorem 4.3]{chodosh2401stable}}]
\label{prop-second}
  Suppose $\Omega \in \mathcal{C}(X)$ is a smooth $\mu$-bubble and let $\Sigma$ be a connected component of $\partial \Omega \setminus \partial_{-} X.$
  For any smooth function $\phi \in C_{0}^{\infty}(\Sigma)$ on $\Sigma$ let $V_{\phi}$ be a vector field on $X,$ which vanishes outside a small neighborhood of $\Sigma$ and agree with $\phi \nu$ on $\Sigma.$ Here, $\nu$ is the outwards pointing unit normal of $\Sigma.$ Let $\Phi_{t}$ be the flow generated by $V_{\phi}$ with $\Phi_{0} = \mathrm{id}.$
  Then
  \begin{equation}
  \label{eq-4}
  \begin{split}
    \left. \frac{d^{2}}{dt^{2}} \right|_{t = 0} \mathcal{A}_{u, h} (\Phi_{t} (\Omega)) &= \int_{\Sigma} \phi^{2} (\Delta^{X} u - \Delta^{\Sigma} u) - 2\phi^{2} u^{-1} g(\nabla^{X} u, \nu)^{2}\, d \mathcal{H}^{n-1} \\
    &+ \int_{\Sigma} u (|\nabla^{\Sigma} \phi|_{g}^{2} - (|A_{\Sigma}|^{2} + \mathrm{Ric}_{g}(\nu, \nu)) \phi^{2})\, d \mathcal{H}^{n-1} \\
    &+ \int_{\Sigma} \phi^{2} g(\nabla^{X} u, \nu) h - \phi^{2} u g (\nabla^{X} h, \nu)\, d \mathcal{H}^{n-1}.
  \end{split}
  \end{equation}
  Here, $\Delta^{X} u, \nabla^{\Sigma} \phi$ and $A_{\Sigma}$ denote respectively the Laplacian of $u$ with respect to $g,$ the gradient of $\phi$ with respect to the induced metric $g|_{\Sigma}$ and the second fundamental form of $\Sigma$ with respect to $g.$ 
\end{prop}

\section{Proofs of Propositions \ref{main-1}, \ref{main-2} and Theorem \ref{main-3}}
\label{section-proof}
In this section we prove Propositions \ref{main-1}, \ref{main-2} and Theorem \ref{main-3}.

We consider the smooth connected compact Riemannian manifold $(X, g|_{X})$ for some $X \subset M,$ and warped $\mu$-bubbles on it.
Using Proposition \ref{prop-existence} we can find a warped $\mu$-bubble $\Omega$ minimizing
\[
\mathcal{A}_{u, h} (\Omega) = \int_{\partial^{*} \Omega} u\, d \mathcal{H}^{n-1} - \int_{X} (\chi_{\Omega} - \chi_{\Omega_{0}}) h u\, d\mathcal{H}^{n}
\]
for all Caccioppoli sets $\Omega$ with $\Omega \Delta \Omega_{0} \Subset \mathring{X}$ for some reference Caccioppoli set $\Omega_{0}.$
And set $\Sigma$ be a component of $\partial \Omega$ contained in $\mathring{X}.$
Then, from Proposition \ref{prop-existence}, $\Sigma$ is compact and smooth.
Moreover, since $\Omega$ is a minimizer of $\mathcal{A}_{u, h},$ its second derivative at $\Omega$ (i.e., the right hand side of (\ref{eq-4})) is nonnegative. 

\smallskip
Now we prepare an auxiliary lemma.
\begin{lemm}[{\cite[p. 13]{chodosh2401stable}}]
  \begin{equation}
\label{eq-5}
\begin{split}
\int_{\Sigma} u |\nabla^{\Sigma} \phi|^{2} - \phi^{2} \Delta^{\Sigma} u
&\le \frac{4}{3} \int_{\Sigma} |\nabla^{\Sigma} (u^{1/2} \phi)|_{g}^{2} \\
&= \frac{1}{3} \int_{\Sigma} u^{-1} \phi^{2} |\nabla^{\Sigma} u|_{g}^{2} + \frac{4}{3} \int_{\Sigma} \phi g(\nabla^{\Sigma} u, \nabla^{\Sigma} \phi) \\
&+ \frac{4}{3} \int_{\Sigma} u |\nabla^{\Sigma} \phi|_{g}^{2}
\end{split}
\end{equation}
for all $\phi \in C^{0, 1}_{0} (\Sigma).$
\end{lemm}
Take $u := e^{\psi f}$ where $\psi$ is a smooth function on $X$ and $f$ is the potential function of the Ricci soliton. 
Then one can easily compute as follows.
 \[
  \begin{split}
   \nabla^{X} u &= u \nabla^{X} (\psi f) = u \psi \nabla^{X} f + u f \nabla^{X} \psi, \\
    \Delta^{X} u &= u \psi \Delta^{X} f + 2u g(\nabla^{X} \psi, \nabla^{X} f) + uf \Delta^{X} \psi + u \psi^{2} |\nabla^{X} f|^{2}_{g} \\
    &~~~+ u f^{2} |\nabla^{X} \psi|_{g}^{2} + 2 u \psi f g(\nabla^{X} \psi, \nabla^{X} f) \\
    &= -u \psi R_{g} + 2u g(\nabla^{X} \psi, \nabla^{X} f) + uf \Delta^{X} \psi + u \psi^{2} |\nabla^{X} f|_{g}^{2} \\
    &~~~+ u f^{2} |\nabla^{X} \psi|_{g}^{2} + 2 u \psi f g(\nabla^{X} \psi, \nabla^{X} f).
  \end{split}
  \]
  We have used the traced soliton equation:$\, R_{g} + \Delta^{X} f = 0$ in the last equality.
  Moreover, from the first variation formula (\ref{eq-3}), we have
  \[
  |A_{\Sigma}|^{2} \ge \frac{1}{n-1} H^{2} = \frac{1}{n-1} |h - u^{-1} g(\nabla^{X} u, \nu)|^{2}.
  \]
  Putting these together into (\ref{eq-4}) and using (\ref{eq-5}), we obtain that
  \begin{equation}\label{eq-2nd-explicit}
  \begin{split}
  0 &\le \int_{\Sigma} -\phi^{2} u \psi R_{g} + \phi^{2} u \psi^{2} |\nabla^{X} f|_{g}^{2} + \phi^{2} u f^{2} |\nabla^{X} \psi|_{g}^{2} + 2 \phi^{2} u \psi f g(\nabla^{X} \psi, \nabla^{X} f) \\
  &+ \int_{\Sigma} -2 \phi^{2} u \psi^{2} g(\nabla^{X} f, \nu)^{2} - 2 \phi^{2} u f^{2} g(\nabla^{X} \psi, \nu)^{2} \\
  &+ \int_{\Sigma} -4 \phi^{2} u \psi f g(\nabla^{X} f, \nu) \cdot g(\nabla^{X} \psi, \nu) + \int_{\Sigma} \phi^{2} u f \Delta^{X} \psi + 2 \phi^{2} u g(\nabla^{X} \psi, \nabla^{X} f) \\
  &- \int_{\Sigma} \frac{\phi^{2} u}{n-1} |h - \psi g(\nabla^{X} f, \nu) - f g(\nabla^{X} \psi, \nu)|^{2} - \mathrm{Ric}_{g}(\nu, \nu) \phi^{2} u \\
  &+ \int_{\Sigma} \phi^{2} u \psi g(\nabla^{X} f, \nu) h + \phi^{2} u f g(\nabla^{X} \psi, \nu) h - \int_{\Sigma} \phi^{2} u g(\nabla^{X} h, \nu) \\
  &+ \frac{1}{3} \int_{\Sigma} \phi^{2} u \psi^{2} |\nabla^{\Sigma} f|_{g}^{2} + \phi^{2} u f^{2} |\nabla^{\Sigma} \psi|_{g}^{2} + \frac{2}{3} \int_{\Sigma} \phi^{2} u \psi f g(\nabla^{\Sigma} \psi, \nabla^{\Sigma} f) \\
  &+ \frac{4}{3} \int_{\Sigma} \phi u \psi g(\nabla^{\Sigma} f, \nabla^{\Sigma} \phi) + \phi u f g(\nabla^{\Sigma} \psi, \nabla^{\Sigma} \phi) + \frac{4}{3} \int_{\Sigma} u |\nabla^{\Sigma} \phi|_{g}^{2}.
  \end{split}
  \end{equation}
  
  Based on these preparations, we first prove Proposition \ref{main-1} and \ref{main-2} in the following two subsections.
\subsection{Proof of Proposition \ref{main-1}}
\label{subsection-1}
We give a proof of Proposition \ref{main-1}.
We prove this separately for the case where $n \le 7$ and for the general case.
\begin{proof}[Proof of Proposition \ref{main-1}]
  
\underline{\textbf{Step 1: smooth case} $\mathbf{(n \le 7)}$}

\smallskip
Fix a point $p \in M.$ Let $0 < L < L'$, and $\rho_{0}$ be a smoothing of the distance function $d_{g}(p, \cdot)|_{\overline{B_{g}(p, L')} \setminus B(p, L)}$ from a fixed point $p \in M$ such that $\rho_{0}(x) = d(p, x)$ for all $x \in \partial \left( \overline{B_{g}(p, L')} \setminus B(p, L) \right)$ and $|\nabla \rho_{0}|_{g} \le 2.$ 
From the Sard's theorem, one can assume that $\partial \left( \overline{B_{g}(p, L')} \setminus B(p, L) \right)$ is smooth compact $(n-1)$-dimensional submanifold of $M.$
Let $X$ be one of the connected components of $\overline{\{ x \in M~|~\rho(x) \le L' \}} \setminus \{ x \in M~|~\rho(x) < L \}.$
Set $\partial_{-} X := \partial \overline{\{ x \in M~|~\rho(x) \le L' \}}$.
Take some reference Caccioppoli set $\Omega_{0}$ with $\partial_{-} X \Subset \Omega_{0} \Subset X \setminus \partial_{+} X,$ and $\psi = \phi \equiv 1$ in the above calculation.
Then, we obtain from (\ref{eq-2nd-explicit}) that
\begin{equation}\label{eq-prop2.1-stability}
\begin{split}
0 &\le -\int_{\Sigma} u R_{g}  + C L^{-\alpha} u + \frac{4}{3} \int_{\Sigma} u |\nabla^{X} f|_{g}^{2} - \frac{n+1}{n-1} u g(\nabla^{X} f, \nu)^{2} \\
&~~~-\frac{1}{n-1} \int_{\Sigma} u \left( h^{2} - (n+1) h g(\nabla^{X} f, \nu) \right) + \int_{\Sigma} u |\nabla^{X} h|_{g}. 
\end{split}
\end{equation}
We have also used the assumption (\ref{eq-ric1}) here.
Using Young's inequality,
\[
(n+1) g(\nabla^{X} f, \nu) h \le \frac{n+1}{2\varepsilon} g(\nabla^{X} f, \nu)^{2} + \frac{\varepsilon (n+1)}{2} h^{2}.
\]
Taking $\varepsilon = \frac{2}{n+1} \delta~(\delta \in (0, 1))$ here, we get
\[
(n+1) g(\nabla^{X} f, \nu) h \le \left( \frac{2}{n+1} \right)^{2} \delta^{-1} g(\nabla^{X} f, \nu)^{2} + \delta h^{2},
\]
and the terms containing $g(\nabla^{X} f, \nu)^{2}$ in the integrand can be estimated as
\[
\begin{split}
\left( \frac{1}{n-1} \left( \frac{n+1}{2} \right)^{2} \delta^{-1} - \frac{n+1}{n-1} \right)\, g(\nabla^{X} f, \nu)^{2} 
&= \frac{(n+1 - 4\delta)(n+1)}{4\delta (n-1)}\, g(\nabla^{X} f, \nu)^{2} \\
&\le \frac{(n+1 - 4\delta)(n+1)}{4\delta (n-1)}\, |\nabla^{X}f|_{g}^{2}.
\end{split}
\]
Putting them together into the above stability inequality (\ref{eq-prop2.1-stability}), we obtain that
\begin{equation}\label{eq-prop2.1-stability2}
\begin{split}
0 \le -\int_{\Sigma}  u R_{g} &+ C L^{-\alpha} u + \left( \frac{4}{3} + \frac{(n+1 - 4\delta)(n+1)}{4\delta (n-1)} \right)\, \int_{\Sigma} u |\nabla^{X} f|_{g}^{2} \\
&-\frac{1-\delta}{n-1} \int_{\Sigma} u h^{2} + \int_{\Sigma} u |\nabla^{X} h|_{g}.
\end{split}
\end{equation}
If $\inf_{M} R(g) = 0,$ then the desired assertion is trivial, hence we can assume that $\inf_{M} R(g) > 0.$
In order to obtain a contradiction, we suppose that
\begin{equation}
\label{eq-6}
-\left( \frac{7}{3} + \frac{(n+1 - 4\delta)(n+1)}{4\delta (n-1)} \right)\, \inf_{M} R(g)
+ \left( \frac{4}{3} + \frac{(n+1 - 4\delta)(n+1)}{4\delta (n-1)} \right)\, C_{0} < 0. 
\end{equation}
Let $L_{0}$ be an arbitrary positive real constant. 
Since the diameter of $(M, g)$ is infinite, we can take $L'$ so that $L_{0} < L' - L$ and neither of the sets $\partial_{-} X, \partial_{+} X$ are non-empty.
Let $\rho : X \rightarrow \left( -\frac{L' - L}{2}, \frac{L' - L}{2} \right)$ be a smoothing of the signed distance function of the set $\{ x \in M~|~d_{g}(\partial_{-} X, x) = d_{g}(\partial_{+} X, x) \}$ with $|\mathrm{Lip}\, \rho| \le 2$ and $\rho \equiv \pm \frac{L' - L}{2}$ on $\partial_{\pm} X.$
Hence, in particular,
\[
\rho \rightarrow
\begin{cases}
-\frac{L' - L}{2} &\mathrm{at}~\partial_{-} X, \\
\frac{L' - L}{2} &\mathrm{at}~\partial_{+} X.
\end{cases}
\]
Take a smooth function $h$ defined as 
\[
h(x) := - \frac{2 \pi A}{L' - L} \tan \left( \frac{\pi}{L' -L} \rho(x) \right),
\]
where $A$ is a real positive constant, which is taken as $A = \frac{n-1}{1-\delta}$.
Then it satisfies
\begin{equation}\label{eq-prop2.1-h}
-\frac{1}{A} h^{2}(x) + |\nabla^{X} h|(x) \le 4A \left( \frac{\pi}{L' - L} \right)^{2}.
\end{equation}
Thus, from (\ref{eq-prop2.1-stability2}) and (\ref{eq-prop2.1-h}) with taking $L$ and $L_{0}$ large enough, we finally obtain that 
\[
0 \le \left. \frac{d^{2}}{dt^{2}} \right|_{t = 0} \mathcal{A}_{u, h} (\Phi_{t} (\Omega)) < 0.
\]
This is impossible.
Therefore our supposition (\ref{eq-6}) was not correct and hence
\[
-\left( \frac{7}{3} + \frac{(n+1 - 4\delta)(n+1)}{4\delta (n-1)} \right)\, \inf_{M} R(g)
+ \left( \frac{4}{3} + \frac{(n+1 - 4\delta)(n+1)}{4\delta (n-1)} \right)\, C_{0} \ge 0.
\]
Sorting this out, we finally obtain that
\[
\frac{\inf_{M} R(g)}{C_{0}} \le A(\delta, n)~~\mathrm{for~all}~\delta \in (0, 1),
\]
where 
\[
A(\delta, n) := \frac{\frac{4}{3} + \frac{(n+1-4\delta)(n+1)}{4\delta(n-1)}}{\frac{7}{3} + \frac{(n+1-4\delta)(n+1)}{4\delta(n-1)}}~~~(< 1).
\]
And it easily turns out that 
\begin{itemize}
\item $A(\delta, n)$ is non-increasing with respect to $\delta,$ and
\item $A(1, n)$ is increasing with respect to $n.$
\end{itemize}
Hence
$\frac{\inf_{M} R(g)}{C_{0}} \le A(1, n) = A(n)$ and $A(\delta, n) > A(1, n) \ge A(1, 2) = \frac{7}{19}.$ 

\underline{\textbf{Step 2: general case} $\mathbf{(n \ge 8)}$}

\smallskip
We prove the general case following the argument in {\cite[Section 4]{bray2019proof}} (cf. {\cite[Appendix]{antonelli2024new}}, {\cite[Proof of Theorem 2.2]{zhou2020existence}}).
Let $\mathcal{S}$ be the singular set of $\Sigma$, then it is compact (since $X$ is compact and $\Omega_0 \Subset X$).
From Proposition \ref{prop-existence}, $\mathcal{S}$ has Hausdorff dimension at most $n-8$. In particular, $\mathcal{H}^{n-7}(\mathcal{S}) = 0$.
Hence, for any $\delta > 0$, there exist a finite collection of balls $\{ B_{r_{i}}(x_{i}) \}$ such that $r_{i} \le \delta$, $\mathcal{S} \subset \cup_{i} B_{r_{i}}(x_{i}) =: S_{\delta}$, and $\sum_{i} r_{i}^{n-7} \le 1$.
For each $i$, one can find a smooth function $\eta_{i} : M \rightarrow [0,1]$ such that
\[
\eta_{i} |_{B_{r_{i}}}(x_{i}) \equiv 0,~~\eta_{i}|_{M \setminus B_{2r_{i}}(x_{i})} \equiv 1,~~|\nabla^{M} \eta_{i}| \le 2 r_{i}^{-1}.
\]
Setting $\tilde{\eta}(x) := \min_{i} \{ \eta_{i}(x) \}$ and regularizing this, we can obtain a smooth function $\eta : M \rightarrow [0,1]$ such that 
\[
\eta|_{\cup_{i} B_{r_{i}(x_{i})}} \equiv 0,~~\eta|_{\cup_{i}B_{M \setminus B_{2r_{i}}(x_{i})}} \equiv 1,~~|\nabla^{M} \eta| \le 2 |\nabla^{M} \tilde{\eta}|,~~\mathcal{H}^{n}-a.e.
\]
Hence, we have
\begin{equation}\label{eq-gradient}
  |\nabla^{M} \eta| \le \sum_{i} 2 |\nabla^{M} \eta_{i}| \le 2 \sum_{i} r_{i}^{-1},~~\mathcal{H}^{n}-a.e.
\end{equation}
Put $S'_{\delta} := \cup_{i} B_{2r_{i}}(x_{i})$ and let $W_{\delta} := (S'_{\delta} \setminus S_{\delta}) \cap \Sigma$.
Since $\Sigma$ is contained in the interior of $X$ and the function $h$ is bounded there, from the first variation formula (\ref{eq-3}), the generalized mean curvature $H$ of $\Sigma$ is bounded.
Therefore, according to the monotonicity formula for varifolds ({\cite[Theorem 17.6]{simon1983lectures}}), we have, for all $i$,
\begin{equation}\label{eq-monotonicity}
  \mathcal{H}^{n-1}(B_{2r_{i}}(x_{i}) \cap \Sigma) \le C r_{i}^{n-1}
\end{equation}
for some positive constant $C$ depending only on $L, L'$ and $g$.
Since $W_{\delta} \subset S'_{\delta} \cap \Sigma$, from (\ref{eq-monotonicity}) we have
\begin{equation}\label{eq-monotonicity2}
\mathcal{H}^{n-1}(W_{\delta}) \le \mathcal{H}^{n-1} (S'_{\delta} \cap \Sigma)
\le \sum_{i} \mathcal{H}^{n-1} (B_{2r_{i}}(x_{i}) \cap \Sigma)
\le C \sum_{i} r_{i}^{n-1}
\le C \delta^{6}.
\end{equation}

By replacing $\phi$ with $\eta \phi$ in Step 1 above, from (\ref{eq-2nd-explicit}), we have
\[
\begin{split}
  0 \le -\int_{U}  u R_{g} &+ C L^{-\alpha} u + \left( \frac{4}{3} + \frac{(n+1 - 4\delta)(n+1)}{4\delta (n-1)} \right)\, \int_{U} u |\nabla^{X} f|_{g}^{2} \\
&-\frac{1-\delta}{n-1} \int_{U} u h^{2} + \int_{U} u |\nabla^{X} h|_{g} \\
&+ C_{1} \mathcal{H}^{n-1}(W_{\delta}) + \frac{4}{3} \int_{W_{\delta}} \eta \phi^{2} u\, g(\nabla^{\Sigma} \eta, \nabla^{\Sigma} f) + \frac{4}{3} \int_{W_{\delta}} u |\nabla^{\Sigma} \eta|_{g}^{2}.
\end{split}
\]
Here, the constant $C_{1}$ depends only on $L, L', g$ and $f$ and $U := (\Sigma \setminus \mathcal{S}) \setminus S'_{\delta}$.
Then, using (\ref{eq-gradient}) and (\ref{eq-monotonicity2}) to estimate the last three terms in the right-hand side, we have
\[
\begin{split}
  0 &\le -\int_{U}  u R_{g} + C L^{-\alpha} u + \left( \frac{4}{3} + \frac{(n+1 - 4\delta)(n+1)}{4\delta (n-1)} \right)\, \int_{U} u |\nabla^{X} f|_{g}^{2} \\
&~~~~~~-\frac{1-\delta}{n-1} \int_{U} u h^{2} + \int_{U} u |\nabla^{X} h|_{g} \\
&~~~~~~~+ C_{1} \delta^{6} + C_{2} r_{i}^{-1}\, \mathcal{H}^{n-1}(W_{\delta}) + C_{3} r_{i}^{-2}\, \mathcal{H}^{n-1}(W_{\delta}) \\
&\le -\int_{U}  u R_{g} + C L^{-\alpha} u + \left( \frac{4}{3} + \frac{(n+1 - 4\delta)(n+1)}{4\delta (n-1)} \right)\, \int_{U} u |\nabla^{X} f|_{g}^{2} \\
&~~~~~~-\frac{1-\delta}{n-1} \int_{U} u h^{2} + \int_{U} u |\nabla^{X} h|_{g} 
+ C_{1} \delta^{6} + CC_{2} \delta^{5} + CC_{3} \delta^{4}.
\end{split}
\]
Therefore, by taking $\delta > 0$ small enough, one can argue in the same way as in Step 1 and get the same conclusion in dimensions $\ge 8$. 
\end{proof}

\subsection{Proof of Proposition \ref{main-2}}
\label{subsection-2}
We next prove Proposition \ref{main-2}.
\begin{proof}[Proof of Proposition \ref{main-2}]
The general case can be proved in the same way as Step 2 in the proof of Proposition \ref{main-1}, so below we prove it for the case where $\Sigma$ is smooth.

Take $\phi \equiv 1$ and $\psi (\cdot) := G(p, \cdot),$ where $G$ is the minimal positive Green's function with the pole at $p.$
Take $X$ as a component of the set $\{ x \in M~|~a \le \psi (x) \le 2a \}~(0 < a \le 1).$
By Sard's theorem and our assumption: $G(x) \rightarrow 0$ as $d_{g}(p, x) \rightarrow \infty$, we can take $a$ so that $X$ is smooth compact connected manifold with boundary.
From the assumption, we can assume that $X \subset M \setminus \overline{B_{g}(p, L)}$ for sufficiently large $L > 0,$ which is determined later.
Take some reference Caccioppoli set $\Omega_{0}$ with $\partial_{-} X \Subset \Omega_{0} \Subset X \setminus \partial_{+} X.$
From the identity (\ref{eq-1}), $|\nabla^{X} f| \le C_{0}^{1/2}$ and
\begin{equation}
\label{eq-7}
  |f(x)| \le C_{0}^{1/2} d_{g}(p,x) + |f(p)|.
\end{equation}
Using these, Proposition \ref{prop-harmonic} and (\ref{eq-ric3}), we obtain from (\ref{eq-2nd-explicit}) that
\[
\begin{split}
0 \le \int_{\Sigma} u \psi &\left( -R_{g} + C (\psi + |f|^{2} \psi L^{-2} + \psi |f| L^{-1} + L^{-1} + L^{-2} \psi \cdot \psi^{-1} ) \right) \\
&+ B(h, |\nabla h|_{g}) 
\end{split}
\]
for some positive constant $C > 0.$
Note here that we have used the assumption (\ref{eq-ric3}) to estimate the $\mathrm{Ric}_{g}(\nu, \nu)$-term. 
Here $B(h, |\nabla h|_{g})$ is the terms containing $h$ or $|\nabla h|$ and this is explicitly expressed as 
\[
B(h, |\nabla h|_{g}) = - \frac{1}{n-1} u h^{2} + \frac{n+1}{n-1} \left( u \psi g(\nabla^{X} f, \nu) h + u f g(\nabla^{X} \psi, \nu) h \right) + u |\nabla^{X} h|.
\]
Suppose that $\inf_{M} R_{g} > 0.$
From the assumption $G(x) \rightarrow 0$ as $d_{g}(p,x) \rightarrow \infty$, we can take $L$ large enough so that $C \psi < \frac{1}{4} \inf_{M} R_{g}$ on $\Sigma.$
Thus, using Young's inequality, we can estimate the term $B(h, |\nabla h|_{g})$ as
\begin{equation}\label{eq-prop2.2-B}
B(h, |\nabla h|_{g}) \le \left( -A \psi^{-1} h^{2} + \psi^{-1} |\nabla h|_{g} + C_{1} \psi g(\nabla^{X} f, \nu)^{2} + C_{2} f^{2} g(\nabla^{X} \psi, \nu)^{2} \right) u \psi
\end{equation}
for some positive constants $A, C_{1}$ and $C_{2}.$
The third and fourth term of the right hand side of (\ref{eq-prop2.2-B}) are estimated by using Proposition \ref{prop-harmonic} and (\ref{eq-7}) so that these are $< \frac{1}{4} \inf_{M} R_{g} u \psi$ on $\Sigma$ by taking $L$ large enough.
Now take a function $h$ defined as 
\[
h(x) := \frac{A^{-1}}{1 - a^{-1} \psi(x)} - \frac{A^{-1} }{1 - (2a)^{-1} \psi(x)}.
\]
Then, we have
\[
\begin{split}
  |\nabla^{X} h|_{g} &\le \frac{(a^{-1} - (2a)^{-1}) |\nabla^{X} \psi| \psi^{2} + (a^{-1} - (2a)^{-1}) |\nabla^{X} \psi| (a^{-1} \psi^{2} + (2a)^{-1} \psi^{2})}{A (1 - a^{-1} \psi)^{2}(1 - (2a)^{-1} \psi)^{2}} \\
  &\le \frac{(2a)^{-1}\frac{C}{L}\psi^{2}}{A(1 - a^{-1} \psi)^{2}(1 - (2a)^{-1} \psi)^{2}}.
\end{split}
\]
for some positive constant $C = C(n) > 0.$
Here, we have used Proposition \ref{prop-harmonic} and $0 < a \le 1$ in the last inequality.
Thus, taking $L$ large enough so that $C/L \le 1,$ we finally obtain that
\[
|\nabla^{X} h|_{g} \le \frac{(2a)^{-1}\psi^{2}}{A(1 - a^{-1} \psi)^{2}(1 - (2a)^{-1} \psi)^{2}}.
\]
On the other hand, since $a \le 1,$
\[
-A h^{2} = - \frac{(2a)^{-2} \psi^{2}}{A(1 - a^{-1} \psi)^{2}(1 - (2a)^{-1} \psi)^{2}}
\le - \frac{(2a)^{-1} \psi^{2}}{A(1 - a^{-1} \psi)^{2}(1 - (2a)^{-1} \psi)^{2}}.
\]
Therefore $-A \psi^{-1} h^{2} + \psi^{-1} |\nabla^{X} h|_{g} \le 0.$
As a result, we come up with the conclusion that
\[
0 \le \left. \frac{d^{2}}{dt^{2}} \right|_{t = 0} \mathcal{A}_{u, h} (\Phi_{t} (\Omega)) < 0,
\]
and conclude the proof by contradiction.
\end{proof}

\begin{rema}\label{rema-2.2}
  From the proof of Proposition \ref{main-2} above, in Proposition \ref{main-2}, the assumption (\ref{eq-ric3}) can be replaced with 
  \[
  \mathrm{Ric}_{g}(x) \ge -C d_{g}^{-2}(p, x)~~\mathrm{and}~\lim_{d_{g}(p, x) \rightarrow +\infty} d_{g}^{2}(p, x) \cdot G(x) = +\infty.
  \]
\end{rema}
\subsection{Proof of Theorem \ref{main-3}}
\label{subsection-3}
Finally, we give a proof of Theorem \ref{main-3}.

\begin{proof}[Proof of Theorem \ref{main-3}]
The general case can be proved in the same way as Step 2 in the proof of Proposition \ref{main-1}, so below we prove it for the case where $\Sigma$ is smooth.

Take $\Omega_{0}, \phi$ and $\psi$ in the same way as in the proof of Proposition \ref{main-2}.
Then we similarly obtain from (\ref{eq-2nd-explicit}) that
\[
\begin{split}
  0 &\le \int_{\Sigma} u \psi d_{g}(p, x)^{-\alpha} \left( -R_{g} d^{\alpha}_{g}(p, x) + 2 \psi^{-1} d^{\alpha}_{g}(p, x) \cdot g(\nabla^{X} \psi, \nabla^{X} f) \right) \\
  &~~~+ \int_{\Sigma} u \psi d_{g}(p, x)^{-\alpha} d^{\alpha}_{g}(p, x) (\mathrm{other~terms}).
\end{split}
\]
Here, by taking $L$ large enough and using (\ref{eq-2}) and (\ref{eq-ric2}), $d^{\alpha}_{g}(p, x) (\mathrm{other~terms})$ can be arbitrary small as in the proof of Proposition \ref{main-2}.
From Proposition \ref{prop-harmonic} and (\ref{eq-1}), the second term can be estimated as 
\[
2 \psi^{-1} d^{\alpha}_{g}(p, x) \cdot g(\nabla^{X} \psi, \nabla^{X} f) 
\begin{cases}
  \le C C_{0}^{1/2} & \mathrm{if}~\alpha = 1, \\
  \le C L^{\alpha - 1} & \mathrm{if}~0< \alpha < 1
\end{cases}  
\]
for some positive constant $0 < C = C(n).$
Therefore, if it was true that 
\[
\liminf_{x \rightarrow \infty} R_{g} d^{\alpha}_{g}(p, x) >
\begin{cases}
C C_{0}^{1/2} &\mathrm{if}~\alpha = 1, \\
0 &\mathrm{if}~0< \alpha < 1
\end{cases}
\]
respectively, 
then we could obtain in both cases that
\[
0 \le \left. \frac{d^{2}}{dt^{2}} \right|_{t = 0} \mathcal{A}_{u, h} (\Phi_{t} (\Omega)) < 0.
\]
This is impossible. 
\end{proof}

\begin{rema}\label{rema-main-3}
  From the proof above, the assumption (\ref{eq-ric2}) can be replaced with 
  \[
  \mathrm{Ric}_{g}(x) \ge -C d^{-2}_{g}(p, x)~~\mathrm{and}~\lim_{d_{g}(p, x) \rightarrow +\infty} d^{2-\alpha}_{g}(p, x) \cdot G(x) = +\infty.
  \]
\end{rema}

\section{Appendix}
\label{section-appendix}
In \cite{bray2019proof}, Bray--Gui--Liu--Zhang found a new proof of Bishop's volume comparison theorem. They used ``singular soap bubble'', and the key idea was investigating the second derivative of the isoperimetric profile function.
Here, we give a proof of another theorem involving Ricci lower bound: Bonnet--Myers theorem, using warped $\mu$-bubbles.
\begin{prop}
Let $n \ge 2$ and $(M^{n}, g, f)$ be a tuple that consists of an $n$-dimensional complete Riemannian manifold $(M, g)$ and a function $f$ on $M$ satisfying
  \begin{equation}
  \label{eq-BE}
    \mathrm{Ric}_{g} + \mathrm{Hess}_{g} f \ge \frac{\lambda}{2} g
  \end{equation}
  for some positive constant $\lambda.$
  Assume that $0 < \sup_{M} |\nabla f|_{g} < \infty.$
  Then $M$ is compact and
  \begin{equation}
  \label{eq-9}
    \mathrm{diam} (M, g) \le \pi \sqrt{\frac{2(n-1)}{\lambda B}} \cdot \frac{B (1+B)^{1/4}}{\left( (2+B) \sqrt{1+B} -2(B+1) \right)^{1/2}},
  \end{equation}
  where
  \[
  B := \frac{(n-1)\lambda}{2 \sup_{M} |\nabla f|_{g}^{2}}.
  \]
\end{prop}
\begin{rema}
  The right hand side of the estimate diverges to $+\infty$ as $B \rightarrow 0$ and converges to $\pi \sqrt{\frac{2(n-1)}{\lambda}}$ as $B \rightarrow +\infty$ respectively.
  In particular, the case of $B = +\infty$ (which corresponds to the case of $f \equiv \mathrm{const}$) recovers Myers' theorem \cite{myers1941riemannian}.
\end{rema}
\begin{proof}
The general case can be proved in the same way as Step 2 in the proof of Proposition \ref{main-1}, so below we prove it for the case where $\Sigma$ is smooth.

  Firstly, note that the 2nd variation of $\mathcal{A}$ can also be written as the following alternative form (see {\cite[Proof of Theorem 4.3]{chodosh2401stable}}):
\[
\begin{split}
0 \le \left. \frac{d^{2}}{dt^{2}} \mathcal{A}(\Omega_{t}) \right|_{t = 0} &= \int_{\Sigma} \psi^{2} \mathrm{Hess}_{g} u (\nu, \nu) - \psi^{2} g(\nabla^{\Sigma} u, \nabla^{\Sigma} \psi) \\
&+ \int_{\Sigma} \psi^{2} g(\nabla^{M} u, \nu) H \\
&-\int_{\Sigma} u \left( \psi \Delta^{\Sigma} \psi + (|A^{\Sigma}|^{2} + \mathrm{Ric}(\nu, \nu))\psi^{2} \right) \\
&- \int_{\Sigma} \psi^{2} g(\nabla^{M} u, \nu)h + \psi^{2} u g(\nabla^{M} h, \nu).
\end{split}
\] 
We prove the proposition by contradiction.
Suppose that
\[
\mathrm{diam}(M, g) > (1 + \tilde{\delta}) \sqrt{\frac{2(n-1)(1+a)^{2} \pi^{2}}{\lambda \varepsilon_{\delta}(1-\delta)}} =: L
\]
for fixed arbitrary constants $\tilde{\delta}, a > 0$
Here, we set
\[
(0,1) \ni \varepsilon_{\delta} := \left( 1 + \frac{(n-1)\lambda \delta}{2 \sup_{M} |\nabla f|_{g}^{2}} \right)^{-1} \frac{(n-1)\lambda \delta}{2 \sup_{M} |\nabla f|_{g}^{2}}~~(\delta \in (0, 1)).
\]
Then there exist two points $p, q \in M$ such that $d_{g}(p, q) = L.$
Consider the component of the set $\overline{B(q, L - \varepsilon_0)} \setminus B(q, \varepsilon_0)$ containing the minimizing geodesic connecting $p$ and $q$, for sufficiently small $\varepsilon_0 > 0$ with $\partial_{-}X = \partial B(p, L - \varepsilon_0) \cap \overline{X}$ and $\partial_{+} X = \partial B(q, \varepsilon_0)$. (More precisely, in order to construct such an $X,$ we need to take a smoothing of the distance function and perturb the boundaries by using Sard's theorem (see the proof of Proposition \ref{main-1}). But, since we can perform such smoothing with arbitrary small error, we can still discuss without loss of generality.)
By Proposition \ref{prop-existence}, we can find a warped $\mu$-bubble $\Omega$ of the functional $\mathcal{A}_{u, h}$ with the choice $u = e^{-f}$ and $\psi = \mathrm{const}$ and some reference Caccioppoli set $\Omega_{0}.$
Set $\Sigma := \partial \Omega \setminus \partial_{-} X.$
Then, by Proposition \ref{prop-second}, (\ref{eq-4}) and our assumption (\ref{eq-BE}), we obtain that
\[
\begin{split}
0 &\le \int_{\Sigma} u \left[ \psi^{2}|\nabla^{M} h| - \frac{\psi^{2}}{n-1} |A^{\Sigma}|^{2} - \frac{\lambda}{2} \psi^{2} \right] \\  
&\le \int_{\Sigma} u \left[ \psi^{2}|\nabla^{M} h| - \frac{\psi^{2}}{n-1} \left( h^{2} + 2g(\nabla^{M} f, \nu) h + g(\nabla^{M} f, \nu)^{2} \right) - \frac{\lambda}{2} \psi^{2} \right].
\end{split}
\]
Using Young's inequality, we obtain that
\[
2g(\nabla^{M} f, \nu) h \le \frac{1}{1-\varepsilon} g(\nabla^{M} f, \nu)^{2} + (1- \varepsilon) h^{2}
\]
for all $\varepsilon \in (0, 1).$
So, 
\[
0 \le \int_{\Sigma} u \left[ \psi^{2} \left( |\nabla^{M} h| - \frac{\varepsilon}{n-1} h^{2} \right) + \frac{\psi^{2}}{n-1} \cdot \frac{\varepsilon}{1-\varepsilon} g(\nabla^{M} f, \nu)^{2} - \frac{\lambda}{2} \psi^{2} \right].
\]
Thus we take $h$ as
  \[
  h(x) := - \frac{(n-1)(1 + a) \pi}{\varepsilon L} \tan \left( \frac{\pi}{L} \rho(x) \right),
  \]
  where $\rho(x)$ is a smooth function that satisfies $\rho \rightarrow -L/2$ at $\partial_{-} X$ and $\rho \rightarrow L/2$ at $\partial_{+} X$ with $|\mathrm{Lip}\, \rho| \le 1 + a.$
  Such a function can be constructed via smoothing the distance function from the equidistant set from $p$ and $q$ with an appropriate modification.
  Then it holds that
\[
 |\nabla^{M} h| - \frac{\varepsilon}{n-1} h^{2} \le \frac{(n-1)(1+a)^{2} \pi^{2}}{\varepsilon L^{2}}.
\]
Then, we get
\[
0 \le \int_{\Sigma} u \psi^{2} \left[ \frac{(n-1)(1+a)^{2} \pi^{2}}{\varepsilon L^{2}} + \frac{1}{n-1} \cdot \frac{\varepsilon}{1-\varepsilon} \sup_{M} |\nabla^{M} f|^{2} - \frac{\lambda}{2} \right].
\]
From the definition of $\varepsilon_{\delta}$, we have
\[
\frac{1}{n-1} \cdot \frac{\varepsilon_{\delta}}{1-\varepsilon_{\delta}}\, \sup_{M} |\nabla^{M} f|^{2} = \frac{\lambda}{2} \delta.
\]
Then it satisfies that
\[
\frac{1}{n-1} \cdot \frac{\varepsilon_{\delta}}{1-\varepsilon_{\delta}}\, \sup_{M} |\nabla^{M} f|^{2} - \frac{\lambda}{2} = -\frac{\lambda}{2} (1-\delta).
\]
Summing up with this and the definition of $L,$ we get a contradiction:
\[
0 \le \int_{\Sigma} u \psi^{2} \left[ \frac{(n-1)(1+a)^{2} \pi^{2}}{\varepsilon_{\delta}} \cdot \frac{1}{L^{2}} + \frac{1}{n-1} \cdot \frac{\varepsilon_{\delta}}{1-\varepsilon_{\delta}} \sup_{M} |\nabla^{M} f|^{2} - \frac{\lambda}{2} \right] < 0.
\]
Hence we obtain that
\[
\mathrm{diam}(M, g) \le (1 + \tilde{\delta}) \sqrt{\frac{2(n-1)(1+a)^{2} \pi^{2}}{\lambda \varepsilon_{\delta}(1-\delta)}}
\]
for all $\tilde{\delta} > 0$ and $a > 0$ by contradiction.
Therefore we finally reach the conclusion that
\[
\mathrm{diam}(M, g) \le \sqrt{\frac{2(n-1)\pi^{2}}{\lambda \varepsilon_{\delta}(1-\delta)}}~~~\mathrm{for~all}~\delta \in (0, 1).
\]
In particular, $M$ is compact.

The right hand side of the previous estimate can be written as 
\[
\pi \sqrt{\frac{2(n-1)}{\lambda B}} \sqrt{\frac{1 + B \delta}{\delta (1 - \delta)}} =: \pi \sqrt{\frac{2(n-1)}{\lambda B}} \cdot F(\delta).
\]
An easy calculation implies that
$F(\delta)$ attains its minimum 
\[
\frac{B(1+B)^{1/4}}{\left( (2+B)\sqrt{1+B} -2(B+1) \right)^{1/2}}
\]
at $\delta = -B^{-1} + B^{-1} \sqrt{1+B}~(\in (0, 1)).$
Hence we obtain the desired estimate.
\end{proof}
\begin{rema}
\label{rema-diam}
This type of estimate is also given by \cite{limoncu2010modifications, tadano2018upper, wu2018myers}.
Among these, {\cite[Theorem 1.3]{wu2018myers}} is the sharpest estimate of this form.
We note here that our estimate (\ref{eq-9}) is not better than that of {\cite[Theorem 1.3]{wu2018myers}}.
Moreover Limoncu \cite{limoncu2010modifications} and Tadano \cite{tadano2018upper} have proven this type of estimate under more general setting (i.e., not only for the case of $(M^{n}, g)$ is a steady gradient Ricci soliton but also for the case of $(M^{n}, g, X)$ is a tuple that consists of a Riemannian $n$-manifold $(M^{n}, g)$ and a vector field $X$ on $M$ satisfying $\mathrm{Ric}_{g} + \frac{1}{2} \mathcal{L}_{X} g \ge \frac{\lambda}{2} g$ for some positive constant $\lambda$ with $\sup_{M} |X|_{g} < \infty$).
\end{rema}
\begin{rema}
  In \cite{antonelli2024new}, Antonelli--Xu gave Bishop--Gromov volume comparison and Bonnet--Myers theorem under a spectral assumption.
  It may be worth asking whether the method used in \cite{antonelli2024new} (``unequally warped $\mu$-bubbles'') can be used to prove Bonnet--Myers theorem when (\ref{eq-BE}) is satisfied in a spectral sense.
\end{rema}

\bigskip
\noindent
\textit{E-mail adress}:~hamanaka1311558@gmail.com

\smallskip
\noindent
\textsc{Department of Mathematics, Graduate School of Science, Osaka University, 1-1 Machikaneyama-cho, Toyonaka, Osaka 560-0043, Japan}

\end{document}